\documentclass[leqno,11pt,twoside]{article}
\usepackage{amsmath, amsthm, amssymb}

\usepackage{geometry}
\usepackage{fancyhdr}
\usepackage{titlesec}

\titleformat{\section}
  {\normalfont\scshape}{\thesection .}{1em}{}
\titleformat{\subsection}
  {\normalfont\scshape}{\thesubsection .}{1em}{}

\title{A Remark on Lazarsfeld's Approach to Castelnuovo-Mumford Regularity}
\date{}
\author{\scshape J\"urgen Rathmann}

\makeatletter
\let\runauthor\@author
\makeatother

\usepackage{mathrsfs}
\usepackage{tikz-cd}
\usepackage{enumitem}
\setlist{nolistsep}

\newtheorem{theorem}{Theorem}[section]
\newtheorem{proposition}[theorem]{Proposition}
\newtheorem{lemma}[theorem]{Lemma}
\newtheorem{corollary}[theorem]{Corollary}

\newtheorem*{ThA}{Main Theorem}

\theoremstyle{remark}
\newtheorem{remark}[theorem]{Remark}
\newtheorem{num}[theorem]{}

\newtheorem{examples}[theorem]{Examples}

\newtheorem{definition}[theorem]{Definition}

\newcommand\cB{\mathscr{B}}
\newcommand\cE{\mathscr{E}}
\newcommand\cF{\mathscr{F}}
\newcommand\cG{\mathscr{G}}
\newcommand\cI{\mathscr{I}}
\newcommand\cO{\mathscr{O}}
\newcommand\bP{\mathbf{P}}
\newcommand\cT{\mathscr{T}}

\newcommand\ooplus{\mathop\oplus\limits}

\DeclareMathOperator{\Ext}{Ext}
\DeclareMathOperator{\Ker}{Ker}
\DeclareMathOperator{\Img}{Im}

\DeclareMathOperator{\codim}{codim}
\DeclareMathOperator{\rk}{rk}
\DeclareMathOperator{\reg}{reg}
\DeclareMathOperator{\len}{length}

\begin{document}
\bibliographystyle{plain}
\maketitle

{\small 
We derive new bounds for the Castelnuovo-Mumford regularity of the ideal sheaf of a 
complex projective manifold of any dimension. They depend linearly on the coefficients of the 
Hilbert polynomial, and are optimal for rational scrolls, but most likely not for other varieties.
Our proof is based on an observation of Lazarsfeld in his approach for surfaces and does
not require the (full) projection step.
We obtain a bound for each partial linear projection of the given variety, as long
as a certain vanishing condition on the fibers of a general projection holds.
}

\section*{Introduction}

Consider a reduced irreducible complex projective variety $X\subset \bP^r$ of dimension $n$, not contained in a hyperplane.
We recall that a sheaf $\cF$ is $m$-regular, if the cohomology groups $H^i(\bP^r,\cF(m-i))$ are zero for $i>0$. 
The smallest such $m$ for the ideal sheaf $\cI_X$ is called the regularity $\reg(X)$ of $X$, 
and it has the following properties:
\begin{enumerate}
\item[(1)] Hypersurfaces of degree $\ge m-1$ cut out a complete linear system on $X$.
\item[(2)] Higher cohomology groups $H^i\cO_X(l)$ ($i>0$) vanish for any $l\ge m-i-1$.
\item[(3)] The homogeneous ideal of $X$ is generated by hypersurfaces of degree $\le m$; more generally, the 
terms of the $l$-th syzygy of the minimal graded resolution of the homogeneous ideal of $X$ have degrees $\ge -m-l$. 
\end{enumerate}

In the early 1980s, Gruson, Lazarsfeld and Peskine \cite{GLP} investigated the regularity of curves and suggested 
that under suitable conditions on $X$ one might have
\begin{equation*}
\reg(X)\le \deg(X)+1-\codim(X).
\end{equation*}

This bound would be sharp. There are examples of smooth rational scrolls 
with a $(\deg(X)+1-\codim(X))$-secant line, at least in the range $r\ge 2\dim(X)+1$ \cite{Laz1}. Such varieties $X$
cannot be cut out by hypersurfaces of degree $\deg(X)-\codim(X)$, hence their ideal sheaf cannot be 
$(\deg(X)-\codim(X))$-regular.

A lot of work focused on smooth varieties $X$ of low dimension, e.g, the conjecture is known 
for surfaces \cite{Laz1} and there exist close bounds for $\dim(X)\le 4$ \cite{Kwak, Ran1}. For further information, see \cite{Laz2} 
and \cite{Kwak}.

The best available bounds for general $X$ of arbitrary dimension involve a multiple of the degree 
(Mumford \cite{BM}: $\reg(X)\le (n+1)(d-2)+2$; Bertram, Ein and Lazarsfeld \cite{BEL}: $\reg(X)\le c(d-1)+1$ with $c=\min(n+1,r-n)$).

Eisenbud and Goto \cite{EisGoto} conjectured that the bound would hold for
arbitrary reduced and irreducible $X$, but this has been disproved by
recent examples of McCullough and Peeva \cite{McCP}.

\medskip
Our contribution is a new sequence of bounds for the regularity of $X$ which depend linearly on the coefficients 
of the Hilbert polynomial of $\cO_X$. They require nonsingularity.

\begin{ThA}
Let $X\subset\bP^r$ be an $n$-dimensional non-degenerate smooth irreducible complex manifold. 
Then we have
\begin{equation*}
\reg(X)\le -(r-m)+\sum_{k=0}^n(-1)^{n+k}\binom{m-1}{n-k}\chi(\cO_X(k+1-n))
\end{equation*}
for any $m$ with
$m_0\le m\le r$\textup{,} where
\begin{equation*}
m_0=\begin{cases}
2 & n=1 \\
\min(r,2n-1) & \text{$n\ge 2$.}
\end{cases}
\end{equation*}
\end{ThA}

To put these bounds into perspective, we list three special cases:
\begin{enumerate}
\item[(i)] If $X$ is a curve of degree $d$ and genus $g$, then
\begin{equation*}
\reg(X)\le d+2+(m-2)g-r.
\end{equation*}
\item[(ii)] If $X$ is a surface of degree $d$, sectional genus $\pi$ and Euler characteristic $\chi$, then
\begin{equation*}
\reg(X)\le d+m(m-3)/2\cdot (\pi-1)-(m-2)(m-3)/2\cdot \chi -(r-m).
\end{equation*}
\item[(iii)] If $X$ is ruled over a curve of genus $g$, embedded as a scroll of degree $d$, then
\begin{equation*}
\reg(X)\le d+(m-1-n)g+n-r+1.
\end{equation*}
\end{enumerate}

In each case, we formally recover the bound of the conjecture by setting $m=1+\dim X$. This is not surprising since
our proof follows the strategy used by Lazarsfeld, which would prove the conjecture in all dimensions, if we knew certain
properties of a general projection into $\bP^{n+1}$ (see \ref{num12} below).

For rational scrolls, our bound does not depend on $m$ and is optimal (a result of Bertin \cite{Bertin};
there is another proof by Kwak and E.~Park \cite{KwakE}). For 
other scrolls, we happen to recover recent bounds of Niu and J.~Park \cite{Niu}, clarifying Bertin's earlier work. 
Both Bertin and Niu-Park analyze the intrinsic geometry of the scroll, in the spirit of \cite{GLP}.

For general varieties, our bound should be compared with the bounds of Bertram-Ein-Lazarsfeld resp.\ Mumford: It 
is linear in the coefficients of the Hilbert polynomial, and a closer inspection reveals that the coefficient 
of the degree $d$ in this sum is always $1$. However, in any concrete example this advantage  
could be more than offset by contributions from the additional terms.

The lower bound for $m$ in Theorem A reflects our knowledge (or rather lack thereof) about the geometry of
the fibers of a general projection and is unlikely to be optimal. In the range $4n/3<r\le 2n$, $m_0$ could be lowered, 
at the expense of a more complicated bound (see (\ref{num19}) below). 

\bigskip
Our proof of Theorem A is based on a cohomological vanishing pattern noted by Lazarsfeld.
It implies that the Beilinson spectral sequence for $\cI_X(1-n)$ provides a complex
\begin{equation*}
0\to a_r\cO(-r)\to a_{r-1}\cO(-r+1)\to \dotsc \to a_1\cO(-1)\to a_0\cO\to 0
\end{equation*}
whose only non-vanishing cohomology sheaf occurs at the position $a_{n+1}\cO(-n-1)$ and is isomorphic to $\cI_X(1-n)$.
Hence there is a short exact sequence
\begin{equation*}\tag{$*$}
0\to\cG(-n-2)\to\cE(-n-1)\to\cI_X(1-n)\to 0
\end{equation*}
where the sheaf $\cG$ is $0$-regular, and the sheaf $\cE$ is part of an exact sequence
\begin{equation*}\tag{$**$}
0\to \cE(-n-1)\to a_n\cO(-n-1)\to a_{n-1}\cO(-n)\to \dotsc \to a_1\cO(-1)\to a_0\cO\to 0.
\end{equation*}
The sequence $(**)$ shows that $\cE$ is locally free, and that its dual is $0$-regular. The latter implies
 a regularity bound for $\cE$ and, via the sequence $(*)$, for $\cI_X$.

It should be obvious that our bounds are not optimal, except for rational scrolls. 
Our motivation for the results in this note was to find an approach to regularity bounds
which avoids studying the fibers of a general projection into $\bP^{n+1}$.

Further progress requires a new idea.

This paper is organized as follows. In section 1 we prove the main result. Section 2 
discusses the geometry underlying the bundles $\cE$ and related approaches. 
In the last section we mention some open problems.

\medskip
I am grateful to Wenbo Niu for his comments on an earlier version of the paper.

\section{Proofs}

Our approach produces regularity bounds for sheaves that fulfill certain cohomological vanishing conditions. The
following definition serves to separate these conditions from the formal computation of the bounds. They will be
verified for ideal sheaves of complex projective manifolds in (\ref{prop13}).

\begin{definition}
Let $\cF$ be a coherent sheaf on $\bP^r$. $\cF$ satisfies property $(C_k)$ if the following two conditions are satisfied:
\begin{enumerate}
\item[(i)] $H^i\cF(-i-1)=0$ for $i< k$.
\item[(ii)] $H^i\cF(-i)=0$ for $i>k$.
\end{enumerate}
\end{definition}

Property $(C_0)$ is the same as being regular. We are interested in sheaves satisfying $(C_k)$ 
for some $k>0$, and we will derive a bound on their regularity.

\begin{num}[Notation for linear projections]
\label{num12}

Let $X\subset\bP^r$ be an irreducible complex manifold of dimension $n$, not contained in a hyperplane, let $\cI_X$
be its ideal sheaf.
Further, let $L\subset\bP^r$ be a linear subspace of dimension $r-m-1$ not meeting $X$. Let
$p\colon P\to\bP^r$ be the blow up of $\bP^r$ in $L$, and $q\colon P\to\bP^m$ be the linear projection from $L$, and 
$f\colon X\to\bP^m$ the restriction of the projection. We assume the following:
\begin{enumerate}
\item[(i)] $f$ is finite,
\item[(ii)] $R^1q_*p^*\cI_X(2)=0$.
\end{enumerate}
The second condition holds in particular if the fibers of the linear projection $f$ have length at most $3$.
\end{num}

The following proposition also applies to $\cI_X(2)$ which can be viewed as the 
degenerate case $L=\emptyset$, $p=q=id_{\bP^r}$.

\begin{proposition}[Lazarsfeld]\label{prop13}
Consider the situation described in \textup{(\ref{num12})}. The sheaf $q_*p^*\cI_X(2)$ has property $(C_{n+1})$.
\end{proposition}

\begin{proof}
The vanishing conditions (i) have essentially been demonstrated by Lazarsfeld \cite{Laz1} as the conditions $(*)$ in the proof
of his Lemma 2.1. He reduces them to $H^0\cI_X(1)=0$ ($X$ is not contained in a hyperplane), $H^1\cI_X=0$ (as $X$ is reduced 
and irreducible) and $H^i\cO_X(-i)=0$ for $1<i<n$ (a consequence of Kodaira vanishing).

Regarding the conditions (ii), we tensor the short exact sequence
\begin{equation*}
0\to q_*p^*\cI_X(2)\to q_*p^*\cO_{\bP^r}(2)\to f_*\cO_X(2)\to 0
\end{equation*}
with $\cO_{\bP^m}(-i)$ and 
find an exact sequence
\begin{equation*}
H^{i-1}f_*\cO_X(2-i)\to H^i \big(q_*p^*(\cI_X(2))\otimes\cO_{\bP^m}(-i)\big) 
\to H^i \big(q_*p^*(\cO_{\bP^r}(2))\otimes\cO_{\bP^m}(-i)\big).
\end{equation*}
As $i-1>n=\dim(X)$, the first term is zero. The third term also vanishes, as
\begin{align*}
q_*p^*(\cO_{\bP^r}(2)) & \cong S^2\big(\cO_{\bP^m}(1)\oplus(r-m)\cO_{\bP^m}\big)\\
& \cong\cO_{\bP^m}(2) \oplus(r-m)\cO_{\bP^m}(1)\oplus \textstyle\binom{r+1-m}{2}\cO_{\bP^m}.
\qedhere
\end{align*}
\end{proof}

\begin{proposition}\label{prop14}
Let $\cF$ be a coherent sheaf on $\bP^r$. If $\cF$ has property $(C_k)$\textup{,} then
\begin{equation*}
H^i\big(\cF(-k)\otimes\Omega^j_{\bP^r}(j)\big)=0
\end{equation*}
for $i\ne k$, any $j$.
\end{proposition}
\begin{proof}
We first claim that
\begin{equation*}
H^i\cF(j)=0 \qquad\qquad \text{for $i<k$, $j\le -i-1$ and for $i>k$, $j\ge -i$.}
\end{equation*}

This can be established with similar arguments as in the proof of \cite[1.8.3(iii)]{Laz2}, using the
canonical Koszul complex on $\bP^r=\bP(V)$:
\begin{equation}\tag{$K_\bullet$}
0\to \wedge^{r+1}V_\bP(-r-1)\to\dotsc\to\wedge^2V_\bP(-2)\to V_\bP(-1)\to\cO_{\bP^r}\to 0
\end{equation}
where $V_\bP=V\otimes\cO_{\bP^r}$.

To establish the desired vanishing, one tensors the truncated sequences
\begin{gather}\tag{$K'_\bullet$}
0\to\Omega^j\to\wedge^j V_\bP(-j)\to \wedge^{j-1} V_\bP(-j+1)\to\dotsc\\
\intertext{(for $i<k$) resp.}
\tag{$K''_\bullet$}
\dotsc \to \wedge^{j+2} V_\bP(-j-2)\to \wedge^{j+1} V_\bP(-j-1)\to \Omega^j\to 0.
\end{gather}
(for $i>k$) with $\cF(j-k)$ and argues similarly. We leave the details to the reader.
\end{proof}

We recall the definition of the Beilinson spectral sequence:

The diagonal $\Delta\subset\bP^r\times\bP^r$ is the zero scheme of a section of the bundle
$\cB=pr_1^*\cO_{\bP^r}(1)\otimes pr_2^*\cT_{\bP^r}$, hence $\cO_\Delta$ is resolved by a Koszul complex
\begin{equation*}\tag{$B_\bullet$}
\dotsc\wedge^2(\cB^\vee)\to\cB^\vee\to\cO_{\bP^r\times\bP^r}\to\cO_\Delta\to 0
\end{equation*}

The Beilinson spectral sequence for a sheaf $\cF$ is the second quadrant spectral sequence associated with the truncated
complex $(B_\bullet^+)$ (which results from $(B_\bullet)$ by replacing $\cO_\Delta$ with $0$) tensored with $pr_2^*\cF$, 
when applying $pr_{1,*}$ \cite[B.1.5]{Laz2}:
{\small
\begin{equation*}
E_1^{pq}= (R^q pr_{1,*})\big(pr_2^*(\cF\otimes \Omega^{-p}(-p))\big)\otimes\cO_{\bP^r}(p) \Longrightarrow 
\begin{cases}
pr_{1,*}\big((pr_2^*\cF)\otimes\cO_\Delta\big)\cong\cF & p+q=0 \\
0 & \text{else.}
\end{cases}
\end{equation*}
}

\begin{theorem}\label{thm15}
Suppose that the coherent sheaf $\cF$ satisfies $(C_k)$. The Beilinson spectral sequence for $\cF(-k)$ provides
a complex
\begin{equation*}
0\to a_r\cO(-r)\overset{\partial_r}\to a_{r-1}\cO(-r+1)\overset{\partial_{r-1}}\to \dotsc \to a_1\cO(-1)\to a_0\cO\to 0
\end{equation*}
whose only non-vanishing cohomology sheaf occurs at the position $a_k\cO(-k)$ and is isomorphic to $\cF(-k)$.
Setting $\cG=\Img(\partial_{k+1})\otimes\cO(k+1)$\textup{,} $\cE=\Ker(\partial_k)\otimes\cO(k)$\textup{,} 
there is an exact sequence
\begin{equation*}\tag{$*$}
0\to\cG(-k-1)\to\cE(-k)\to\cF(-k)\to 0.
\end{equation*}
The sheaf $\cG$ is $0$-regular\textup{,} and the sheaf $\cE$ is locally free and is part of an exact sequence
\begin{equation*}\tag{$**$}
0\to \cE(-k)\to a_k\cO(-k)\to a_{k-1}\cO(1-k)\to \dotsc \to a_1\cO(-1)\to a_0\cO\to 0.
\end{equation*}
\end{theorem}
\begin{proof}
Using proposition \ref{prop13}, we compute
\begin{align*}
E_1^{pq} &= H^q\big(\cF(-k)\otimes\Omega^{-p}(-p)\big)\otimes\cO_{\bP^r}(-p) \\
&=
\begin{cases}
H^k\big(\cF(-k)\otimes\Omega^{-p}(-p)\big)\otimes\cO_{\bP^r}(-p) & \text{for $q=k$}\\
0 & \text{for $q\ne k$}
\end{cases}
\end{align*}
Hence $E_1^{pq}$ has nonzero entries only in the row $q=k$, and this row together 
with the corresponding differentials yields the complex of the theorem. The 
remaining statements of the theorem follow immediately.
\end{proof}

Lazarsfeld's approach in \cite{Laz1} implicitly uses the special case $r=k=3$, where $\cG=0$ and $\cE=\cF=q_*p^*\cI_S(2)$.

For future reference, we denote $a_i=\dim H^k\big(\cF(-k)\otimes\Omega^i(i)\big)$.

\begin{remark}\label{rmk16}
For an $n$-dimensional smooth complex manifold $X$, we set $\cF=\cI_X(2)$ and obtain a ``canonical'' locally free resolution of $\cI_X$:
\begin{equation}\label{eq16a}\tag{$R_\bullet$}
0\to a_r\cO(-r)\to \dotsc \to a_{n+2}\cO(-n-2)\to \cE(-n-1)\to \cI_X(1-n)\to 0.
\end{equation}
Its dual provides a ``canonical'' locally free resolution of $\omega_X$
\begin{equation*}
0\to \cO(n-1)\to\cE^\vee(n+1) \to a_{n+2}\cO(n+2)\to\dotsc\to
a_r\cO(r)\to\omega_X(r+n)\to 0.
\end{equation*}

The last sequence shows that $\omega_X(n)$ has a resolution with $r-n-1$ 
linear steps in the given embedding. In particular, $\omega_X(n)$ is globally 
generated for any proper nondegenerate projective submanifold $X\subset\bP^r$ (i.e.,
excluding $(X,\cO_X(1)=(\bP^n,\cO(1)$).
This recovers a result of Ein \cite{Ein} which is used in Ein and Lazarsfeld's 
investigation of syzygies of smooth projective varieties of 
arbitrary dimension \cite[p.\,59]{EL}.
\end{remark}

\begin{corollary}\label{cor17}
Assume that $\cF$ satisfies $(C_k)$\textup{,} and let $\cE$ be the vector bundle associated 
with $\cF$ in the previous theorem. Then we have
\begin{equation*}
\reg(\cF)=\reg(\cE) \le -c_1(\cE) =-\sum_{i=1}^k (-1)^i i a_{k-i}= \sum_{j=0}^{k-1} (-1)^j \binom{r-1}{k-1-j}b_{j-k}
\end{equation*}
where $b_i=(-1)^k \chi\big(\cF(i)\big)$.
\end{corollary}
\begin{proof}
The first equality follows from the $0$-regularity of the sheaf $\cG$ in $(*)$. 

Regarding the inequality, we first of all note that the dual sequence to $(**)$ shows that $\cE^\vee$ is $0$-regular. According
to Lazarsfeld \cite[2.7]{Laz1}, $\wedge^j\cE^\vee$ is then also $0$-regular for any $j$. We now use that if $\cE$ has rank $m$, 
then the non-degenerate pairing $\cE\otimes\wedge^{m-1}\cE\to\det(\cE)$ implies that $\cE\cong\wedge^{m-1}\cE^\vee\otimes\det(\cE)$,
hence $\cE$ is $l$-regular with $l=-c_1\det(\cE)=-c_1\cE$.

$c_1(\cE)$ is computed from the exact sequence $(**)$.

Regarding the last equality, we first note that
\begin{equation*}
a_i=\dim H^k\big(\cF(-k)\otimes\Omega^i(i)\big)=(-1)^k \chi\big(\cF(-k)\otimes\Omega^i(i)\big)
\end{equation*}
as all other cohomology groups of $\cF(-k)\otimes\Omega^i(i)$ vanish by (\ref{prop13}).

We now tensor the complex
\begin{equation*}\tag{$K'_\bullet$}
0\to\Omega^i\to\wedge^i V_\bP(-i)\to\wedge^{i-1}V_\bP(-i+1)\to\dotsc\to 
V_\bP(-1)\to\cO_{\bP^r}\to 0
\end{equation*}
with $\cF(i-k)$, take Euler characteristics and find
\begin{equation*}
\chi\big(\cF(-k)\otimes\Omega^i(i)\big) = \sum_{j=0}^i (-1)^j \binom{r+1}{i-j} \chi\big(\cF(j-k)\big).
\end{equation*}

Therefore
\begin{align*}
c_1(\cE) & = \sum_{i=1}^k (-1)^i i a_{k-i} \\
& = \sum_{i=1}^k \sum_{j=0}^{k-i}(-1)^{i+j} i \binom{r+1}{k-i-j} b_{j-k} \\
&= \sum_{j=0}^{k-1} (-1)^j \sum_{i=1}^{k-j} (-1)^i i \binom{r+1}{k-i-j} b_{j-k}
\end{align*}

Setting $l=k-j$, the coefficient of $b_{j-k}$ is
\begin{equation*}
(-1)^j \sum_{i=1}^l (-1)^i i \binom{r+1}{l-i}
\end{equation*}
which can be viewed as the coefficient of $t^l$ in the product of the formal power series
\begin{equation*}
f(t)=\sum_{i=0}^\infty (-1)^i i t^i=t\cdot\frac{d}{dt} \Big(\sum_{i=0}^\infty (-1)^i t^i \Big)
=t\cdot\frac{d}{dt} \Big(\frac{1}{1+t}\Big)
=-\frac{t}{(1+t)^2}
\end{equation*}
and
\begin{equation*}
g(t)=(-1)^j (1+t)^{r+1},
\end{equation*}
hence agrees with the coefficient of $t^l$ in $(f\cdot g)(t)=(-1)^{j+1} t(1+t)^{r-1}$ as claimed.
\end{proof}

\begin{proof}[Proof of the Main Theorem]
We consider the situation described in (\ref{num12}): The sheaf $q_*p^*\cI_X(2)$ has property $(C_{n+1})$, and 
corollary (\ref{cor17}) provides a bound for its regularity in terms of the Euler characteristics
$\chi\big(q_*p^*\cI_X(2)\otimes\cO_{\bP^m}(i-2)\big)$ for $i=-n+1,\dotsc,0,1$.

Setting $b'_i= \chi(\cO_X(i))$, and considering $b_i=(-1)^{n+1}\chi\big(q_*p^*\cI_X(2)\otimes\cO_{\bP^m}(i-2)\big)$,
the short exact sequence
\begin{equation*}
0\to q_*p^*\cI_X(2)\otimes\cO_{\bP^m}(j-2)\to q_*p^*\cO_{\bP^r}(2)\otimes\cO_{\bP^m}(j-2)\to f_*\cO_X(j)\to 0
\end{equation*}
shows that
\begin{equation*}
(-1)^{n+1} b_j+ b'_j=\chi\big(q_*p^*\cO_{\bP^r}(2)\otimes\cO_{\bP^m}(j-2)\big)=
\begin{cases}
r+1  & \text{for $j=1$}\\
1 & \text{for $j=0$} \\
0 & \text{for $-m\le j<0$}
\end{cases}
\end{equation*}

Hence \cite[1.5]{Laz1} (for the first inequality) together with (\ref{cor17}) (for the second inequality) imply that
\begin{align*}
\reg(\cI_X(2)) &\le \reg\big(q_*p^*\cI_X(2)\big)\le \sum_{j=0}^n (-1)^j \binom{m-1}{n-j}b_{j-n+1} \\
&= (m-1) -(r+1) +\sum_{j=0}^n (-1)^{n+j+1} \binom{m-1}{n-j}b'_{j-n+1}.
\end{align*}
Regarding the range for $m$, we refer to (\ref{num19}) below.
\end{proof}

\begin{examples}
$X\subset\bP^r$ denotes a non-degenerate smooth irreducible complex manifold. We parametrize the Hilbert polynomial 
of $X$ as follows:
\begin{equation*}
\chi(\cO_X(z))=\sum_{j=0}^n c_j\binom{z+j-1}{j}.
\end{equation*}
We note that $c_n=\deg X$ and $c_j(X)=\chi(\cO_{X\cap H_j})$ where $H_j$ is a 
general linear space of codimension $j$.

We further assume that $X$ can be projected linearly into $\bP^m$ as described in (\ref{num12}).
\begin{enumerate}
\item[1.] Let $X$ be a curve of degree $d$ and genus $g$. This implies $c_1=d$, $c_0=1-g$, and we calculate
\begin{equation*}
\reg(X)\le d+2+(m-2)g-r
\end{equation*}
\item[2.] Let $X$ be a surface of degree $d$, sectional genus $\pi$ and Euler characteristic $\chi$. This implies
$c_2=d$, $c_1=1-\pi$, $c_0=\chi$, and
\begin{equation*}
\reg(X)\le d+m(m-3)/2\cdot (\pi-1)-(m-2)(m-3)/2\cdot \chi -(r-m)
\end{equation*}
\item[3.] Let $X$ be ruled over a curve of genus $g$, embedded as a scroll with degree $d$. As a general intersection with 
a linear space is again a scroll over a curve of genus $g$,  we find $c_n=d$, $c_{n-1}=\dotsc=c_0=1-g$, and
\begin{equation*}
\reg(X)\le d+(m-1-n)g+n-r+1
\end{equation*}
\end{enumerate}
\end{examples}
\begin{proof}
\begin{enumerate}
\item[1.] Curves: $c_1=d$, $c_0=1-g$ imply $b'_1=c_1+c_0$, $b'_0=c_0$. Therefore
$-c_1(\cE)=(m-r)-2-\binom{m-1}{1}b'_0+\binom{m-1}{0}b'_1=%
c_1-(m-2)c_0-(r-m)-2=d+(m-2)(g-1)-(r-m)-2%
=d+(m-2)g-m+2-r+m-2=d+(m-2)g-r$.
\item[2.] Surfaces: $c_2=d$, $c_1=1-\pi$, $c_0=\chi$ imply $b'_1=c_0+c_1+c_2$, $b'_0=c_0$, $b'_{-1}=-c_1+c_0$.
Therefore
$-c_1(\cE)=m-r-2+\binom{m-1}{2}b'_{-1}-\binom{m-1}{1}b'_0+\binom{m-1}{0}b'_1
=c_2+\big(-\binom{m-1}{2}+\binom{m-1}{0}\big)c_1+\big(\binom{m-1}{2}-\binom{m-1}{1}+\binom{m-1}{0}\big)c_0-(r-m)-2
=d+m(m-3)/2\cdot(\pi-1)-(m-2)(m-3)/2\cdot\chi-(r-m)-2$.
\item[3.] Scrolls over a curve: $c_n=d$, $c_{n-1}=\dotsc=c_0=1-g$ imply that
$\chi\cO_X(z)=d\binom{z+n-1}{n}+\sum_{j=0}^{n-1}(1-g)\binom{z+j-1}{j}
=d\binom{z+n-1}{n}+(1-g)\binom{z+n-1}{n-1}$. Accordingly, $b'_j=\chi\cO_X(z)$ vanishes for $-1\ge z\ge 1-n$,
$b'_1=d+n(1-g)$ and $b'_0=1-g$. This leads to
$-c_1(\cE)=(m-r)-2-\binom{m-1}{1}b'_0+\binom{m-1}{0}b'_1=
d+n(1-g)-(m-1)(1-g)-(r-m)-2=
d+(m-1-n)(g-1)-(r-m)-2=
d+(m-1-n)g+n-r-1$. \qedhere
\end{enumerate}
\end{proof}

\begin{num}\label{num19}
Consider the situation described in (\ref{num12}), i.e, $X\subset\bP^r$ is an irreducible nondegenerate projective manifold,
$f\colon X\to\bP^m$ the projection from a general linear subspace $L$.

Ran showed \cite[5.6]{Ran2} that the locus of fibres of $f$ of length $k$ or more has codimension at least $k(m-n)$ in $\bP^m$
under either of the following conditions:
\begin{enumerate}
\item[(1)] $\dim L\le 1$,
\item[(2)] $\dim L<r-n+\min(2-n/3,0)$.
\end{enumerate}

This implies that a general projection has no fibers of length $\ge 4$, if 
$m>4n/3$ and one of the following conditions holds:
\begin{enumerate}
\item[(i)] $m=r-1$;
\item[(ii)] $m=r-2$; 
\item[(iii)] $m>2n-r+\max(n/3-2,0)$.
\end{enumerate}

If all fibers of the projection have length at most $3$, then $R^1q_*p^*\cI_X(2)$ vanishes, 
and the projection approach can be applied to bound the regularity of $X$. 

Theorem A only uses the special case of a projection from $\bP^{2n+1}$ into $\bP^{2n-1}$. 
Given $X\subset\bP^r$ and $r\ge 2n+1$, $X$ can be first projected isomorphically into $\bP^{2n+1}$, 
and from there into $\bP^{2n-1}$. This corresponds to $r=2n+1$ and case (ii) above.
\end{num}

\begin{num}
Considering the regularity bound as a linear function in the $c_j$, the coefficient of the highest 
power of $z$, $c_n=\deg(X)$, only contributes to the term $\chi(\cO_X(1))$ corresponding to 
$k=n$, hence appears in the bound with multiple $(-1)^{2n}\binom{m-1}{0}=1$.
\end{num}

\section{Geometry of the bundle $\cE$, and related approaches}

\begin{num}[The splitting type of $\cE$]
A vector bundle $\cE$ on $\bP^r$ splits over any line $L$ as a direct sum 
$\cE_L=\cE\otimes\cO_L\cong\oplus_i \cO_L(a_i)$. The numerical type of 
the splitting and the corresponding geometry of the lines of a particular splitting type 
are interesting invariants of $\cE$.

The vector bundles $\cE$ constructed in (\ref{thm15}) are subbundles of a trivial bundle, hence
$a_i\le 0$ for all $i$. In addition $\sum_i a_i=c_1\cE$, thus $a_i=c_1\cE-\sum_{j\ne i}a_j\ge c_1\cE$. 
If there is a line $L$ where the restriction of $\cE$ has a direct summand $\cO(c_1\cE)$, 
then clearly $\reg\cE=-c_1\cE$.

In a geometric situation, we have much stronger restrictions. Our calculation 
did not use any of these, but they should eventually translate into better regularity bounds.

It turns out that the splitting type of a vector bundle $\cE$ derived from the ideal sheaf of a 
projective manifold $X$ reflects the geometry of the multisecant lines of $X$ which
contain at least 4 points (with multiplicity).
\end{num}

\begin{proposition}
Let $\cE$ be the vector bundle corresponding to the ideal sheaf of a projective manifold $X$ 
as in \textup{(\ref{thm15})}\textup{,}
and let $\cE_L=\cE\otimes\cO_L\cong\oplus_i \cO_L(a_i)$ be its splitting type over a line $L$.
\begin{enumerate}
\item[\textup{(i)}]
$(r-n+1)-d\le a_i\le 0$ for all $i$.
\item[\textup{(ii)}] We have
\begin{equation*}
\cI_X\otimes\cO_L\cong
\begin{cases}
N_{X/\bP^r}^\vee\otimes\cO_L & \text{if $L\subset X$} \\
\cO_L(-l) & \text{if $L\not\subset X$ and $l=\len(\cO_{X\cap L})$.}
\end{cases}
\end{equation*}
\item[\textup{(iii)}]
The kernel of the epimorphism $\cE_L\to\cI_X(2)\otimes\cO_L$ is a direct sum
of copies of $\cO_L$ and $\cO_L(-1)$.
The number of terms of $\cO_L(-1)$ in the sum is determined by the requirement
that $\sum_i a_i=c_1\cE$.
\item[\textup{(iv)}] If $L\not\subset X$ and $l=\len(\cO_{X\cap L})\ge 2$\textup{,} then the
epimorphism $\cE_L\to\cI_X(2)\otimes\cO_L=\cO_L(2-l)$ splits.
\end{enumerate}
\end{proposition}
\begin{proof}
(i) As $\cE$ is a subbundle of a trivial bundle, we have $a_i\le 0$ for all $i$.
Lemma (\ref{lem23}) shows (via induction) that the intersection of $X$ with a general linear space 
of dimension $m=r-n+2$ containing $L$ will be a smooth irreducible nondegenerate surface 
$S\subset\bP^m$ with the same degree and codimension.

The sequence
\begin{equation}\tag{$R_\bullet$}
0\to a_r\cO(-r)\to \dotsc \to a_{n+2}\cO(-n-2)\to \cE(-n-1)\to \cI_X(1-n)\to 0.
\end{equation}
from (\ref{rmk16}) remains exact after restricting to $\bP^m$. As
$\cI_{S/\bP^m}\big(d-(r-n-1)\big)$ is globally generated by Lazarsfeld's solution
of the surface case, we conclude that $\cE\big(d-(r-n+1)\big)\otimes\cO_{\bP^m}$
is also globally generated, hence the same holds for $\cE_L\big(d-(r-n+1)\big)$.

(ii) If $L\subset X$, then
\begin{equation*}
\cI_X\otimes\cO_L\cong(\cI_X\otimes\cO_X)\otimes_{\cO_X}\cO_L\cong N_{X/\bP^r}^\vee\otimes\cO_L.
\end{equation*}
If $L\not\subset X$, then $\cI_X\otimes\cO_L$ is a subsheaf of $\cO_L$, and the isomorphism 
follows from the exact sequence
\begin{equation*}
0\to\cI_X\otimes\cO_L\to\cO_L\to\cO_{X\cap L}\to 0.
\end{equation*}

(iii) Restricting \textup{(\ref{eq16a})} to $L$, we obtain an exact sequence
\begin{equation*}
a_{n+2}\cO_L(-1) \to \cE_L \to \cI_X(2)\otimes\cO_L \to 0.
\end{equation*}
The kernel of the map on the right is torsionfree, hence locally free. By exactness it
is generated in degree $1$ and it embeds into a trivial bundle. This excludes all possibilities
for direct summands except $\cO_L$ and $\cO_L(-1)$.

(iv) The exact sequence
\begin{equation*}
0\to n_1 \cO_L(-1) \oplus n_2\cO_L\to\cE_L\to\cO_L(2-l)\to 0
\end{equation*}
splits because
\begin{equation*}
\Ext^1\big(\cO_L(2-l),n_1 \cO_L(-1) \oplus n_2\cO_L\big)=n_1 H^1\cO_L(l-3)\oplus n_2 H^1\cO_L(l-2)=0
\end{equation*}
for $l\ge 2$.

We note that the term $\cO_L(2-l)$ will only stand out in the splitting of $\cE_L$,
if $2-l\le -2$, i.e., if the line $L$ meets $X$ in at least 4 points. 
\end{proof}

The lower bound for $a_i$ in part (i) would follow immediately, if we knew that the 
twisted ideal sheaf $\cI_X(d+1+n-r)$ is globally generated. 
Noma has investigated a related stronger condition, whether a projective manifold can be cut out by cones of degree $d+1+n-r$, 
see \cite{Noma}.

\begin{lemma}\label{lem23}
Let $X\subset\bP^r$ be an irreducible nondegenerate projective manifold 
of dimension $n\ge 2$\textup{,} and let $L\subset\bP^r$ be a linear subspace 
of dimension $l$ that intersects $X$ in a subscheme of dimension $\le s$. 

If $n>s+l$\textup{,} then the intersection $H\cap X$ with a general 
hyperplane $H$ containing $L$ is also smooth and irreducible.
\end{lemma}
\begin{proof}
Consider a point $x\in X$. The intersection $H\cap X$ will be nonsingular in $x$, if $H$ intersects the embedded 
tangent space $T_x$ transversely, i.e., if $H$ does not contain $T_x$.

Hyperplanes containing $L$ form a family of dimension $r-l-1$, and the subfamily containing $T_x$ for
a given point $x\in L\cap X$ has
codimension $\dim T_x/(L\cap T_x)\ge n-l$. As $x$ varies in a family of dimension $s$, we find that 
the subfamily failing to intersect $X$ transversely, has codimension at least $\dim T_x/(L\cap T_x)-s\ge n-l-s$.

Conversely, if $n-l-s>0$, then a general hyperplane containing $L$ will intersect $X$ transversely in all
points of $L\cap X$, hence will be nonsingular along $L\cap X$.

Nonsingularity of $H\cap X$ outside $L$ and connectedness now follow from Bertini's theorem.
\end{proof}

\begin{num}[Beilinson free monads]
Eisenbud, Fl\o ystad and Schreyer  \cite[8.11]{EFS} constructed a free monad corresponding 
to the Beilinson spectral sequence for arbitrary coherent sheaves. 
It consists of a complex
\begin{equation*}
0\to\cE_{-m}\to\cE_{-m+1}\to\dotsc\to\cE_m\to 0
\end{equation*}
where
\begin{equation*}
\cE_i=\ooplus_{p=0}^m H^{i-p}\big(\Omega^p(p)\otimes q_*p^*\cI_X(l)\big)\otimes\cO(-p)
\end{equation*}
with similar properties as our complex in (\ref{thm15}), i.e., the complex is exact at all positions except at 
$\cE_0$, and the homology at this term is isomorphic to the given sheaf.

However, a concrete bound requires the knowledge of the cohomology 
groups $H^q\big(\cF\otimes\Omega^p(p)\big)$, and not just of the Hilbert polynomial. Key advantage of our 
choice is that, for each $p$, all the cohomology $H^j\big(\Omega^p(p)\otimes q_*p^*\cI_X(1-n)\big)$ groups vanish
except the one for $j=n+1$.

One can also consider monads corresponding to the second Beilinson spectral 
sequence which arises from the bundle $\cB'=pr_1^*\cT_{\bP^r}\otimes pr_2^*\cO_{\bP^r}(1)$.
\end{num}

\begin{num}[Eliminating obstructions by extension]\label{num25}
Suppose that the sheaf $\cF$ satisfies $(C_k)$. If $\cF(1)$ also satisfies $(C_k)$, 
then the Beilinson spectral sequence for $\cF(1-k)$ will deliver a better 
bound for the regularity of $\cF$ than the sequence for $\cF(-k)$.

The obstructions to a lower bound are the groups $H^i\cF(-i)$ for $i<k$. 
Starting with the associated vector bundle $\cE$ from Theorem 
(\ref{thm15}), we can construct another vector bundle $\cE_1$ 
together with an epimorphism $f\colon\cE_1\to\cE$
with the following properties:
\begin{enumerate}
\item[(i)] $f$ induces an isomorphism $H^i \cE_1(l)\to H^i\cE(l)$ on all 
intermediate cohomology groups for $1\le i\le r-1$ and $l>-i$.
\item[(ii)] $H^i\cE_1(-i)=0$ for all $i<r$, i.e., $\cE_1^\vee$ is $(-1)$-regular.
\end{enumerate}

To find $\cE_1$, one uses descending induction on $i$. One easily checks that
\begin{align*}
H^i\cE(-i)' & \cong \Ext^{n-i}(\cE(-i),\cO(-r-1)) \\
& \cong \Ext^1\big(\cE(-i),\wedge^{r-i-1}\cT\otimes\cO(-r-1)\big) \\
& \cong \Ext^1\big(\cE, \wedge^{r-i-1}\cT\otimes \cO(i-r-1)\big)
\end{align*}
and replaces $\cE$ successively by the corresponding extensions.

For details, see \cite{Rat}.
\end{num}

\begin{num}[Cubics]
Working with $\cI_X(3)$ instead of $\cI_X(2)$ changes the requirements and the 
results of our calculation as follows:
\begin{enumerate}
\item[(i)] The projection requires $f$ to be finite and the weaker vanishing of 
$R^1q_*p^*\cI_X(3)$. This means that the projection requires only that 
$H^1\cI_{X\cap L}(3)$ vanishes for each fiber of $f$. Hence we obtain bounds 
for a wider range of $m$.
\item[(ii)] In line with (\ref{num25}), the regularity bound changes by
\begin{equation*}
-\rk\cE+h^0\cI_X(2)+ \textstyle\binom{m-1}{1}h^1\cI_X(1)+ \textstyle\binom{m-1}{2}h^1\cO_X.
\end{equation*}
\end{enumerate}

The theorems of Zak \cite[3.4.25]{Laz2} and Barth \cite[3.2.1]{Laz2} provide numerical conditions 
for the vanishing of $H^1\cI_X(1)$ and $H^1\cO_X$ in low codimension, leading to improved bounds.

Further improvements along the same line require conditions for the vanishing of
$H^1\cI_X(2)$, $H^1\cO_X(1)$ and $H^2\cO_X$, but these are only known for the
last of the three terms.
\end{num}

\section{Open problems}

\begin{num}[The advantage of projections]
The regularity is a measure for the size of the higher cohomology modules $\oplus_l H^i\cI_X(l)$, $(i\ge 1)$ of ideal shaves.
If we project $X$ into a lower-dimensional projective space, then the higher cohomology modules of $q_*p^*\cI_X$, 
as complex vector spaces, do not change (for $i\ge 2$) or possibly even increase (for $i=1$). Hence the regularity of the
projected sheaf $q_*p^*\cI_X(2)$ can be no better than the regularity of $\cI_X(2)$.

However, in our examples (and probably in general) the calculated regularity bound improves under projection.

Is there a good explanation for this counter-intuitive behavior? 
\end{num}

\begin{num}[Linear normality]
Let $X\subset\bP^r$ as usual with ideal sheaf $\cI_X$, and suppose that $X$ arises by linear projection from $X'\subset\bP^{r'}$
with corresponding ideal sheaf $\cI_{X'}$.

For $i\ge 2$, we have $H^i\cI_X(l)\cong H^{i-1}\cO_X(l)$, and these cohomology groups do not depend on whether the embedding
is linearly normal. 

For $i=1$, the short exact sequence
\begin{equation*}
0\to \cI_X(1)\to q_*p^*\big(\cI_{X'}(1)\big) \to \cO_{\bP^r}^{\oplus (r'-r)}\to 0
\end{equation*}
shows:
\begin{enumerate}
\item[(i)] The induced map of graded $SV$-modules 
\begin{equation*}
\ooplus_l H^1\cI_X(l) \to \ooplus_l  H^1 \big(q_*p^*\big(\cI_{X'}(1)\big)\otimes\cO_{\bP^r}(l-1)\big)
\end{equation*}
is surjective.
\item[(ii)] The kernel of this map (which we denote $\oplus_l K(l)$) is generated by $r'-r$ elements of degree $1$.
\end{enumerate}

Hence it suffices to obtain vanishing bounds for $\oplus_l K(l)$ and for 
$\oplus_l  H^1 \big(q_*p^*\big(\cI_{X'}(1)\big)\otimes\cO_{\bP^r}(l-1)\big)$. 

We wonder whether the following hold:
\begin{enumerate}
\item[1.] $K(l)=0$ for $l\ge r'-r+1$,
\item[2.] $H^1 \big(q_*p^*\big(\cI_{X'}(1)\big)\otimes\cO_{\bP^r}(l-1)\big)=H^1\cI_{X'}(l)$.
\end{enumerate}

If true this would reduce the conjecture to the case of linearly normal projective embeddings.
\end{num}

\begin{num}[Linear projections]
Suppose $X\subset\bP^r$ is projected from a linear subspace $L$ into $\bP^m$, and 
choose coordinates in $\bP^r$ such that $L$ is defined by $T_0=T_1=\dotsc=T_{r-m-1}=0$.

If the map of sheaves
\begin{equation*}
q_*p^*\cI_X(2)\otimes\cO_{\bP^m}(l-2)\to f_*\cO_X(l)
\end{equation*}
is surjective on global sections for $l\gg 0$, then every section of $\cO_X(l)$ is induced 
by a hypersurface of $\bP^r$ of the form
\begin{equation*}
\sum P_{ij} T_i T_j+\sum Q_i T_i +R \qquad (r-m\le i,j\le r)
\end{equation*}
for this fixed set of coordinates, where $P_{ij}\in H^0\cO_{\bP^m}(l-2)$, $Q_i\in H^0\cO_{\bP^m}(l-1)$ 
and $R\in H^0\cO_{\bP^m}(l)$ (see \cite[1.5]{Laz1}).

This assumption may be too optimistic, in particular for $X$ of low codimension. 
A more promising assumption might be the following:
\medskip
\begin{enumerate}
\item[(A)] For every section of $\cO_X(l)$, there exist coordinates in $\bP^r$ such that the 
section is induced by a hypersurface of the form above.
\end{enumerate}
\medskip
To investigate this question, one should study all linear projections simultaneously.
\end{num}

\bigskip
\textsc{21 rue de Saussure, 75017 Paris, France}

\medskip
\textit{E-mail}: {\tt jk.rathmann@gmail.com}


\begin{thebibliography}{99}
\small
\bibitem{BM} David Bayer and David Mumford, \textit{What can be computed in algebraic geometry}?
in: Computational Algebraic Geometry and Commutative Algebra (Cortona 1991), 
Cambridge University Press, Cambridge 1993, 1--48.

\bibitem{Bertin} Marie-Amélie~Bertin, \textit{On the regularity of varieties having an extremal secant line}, 
J.~Reine~Angew.~Math.~\textbf{545} (2002), 167--181.

\bibitem{BEL} Aaron~Bertram, Lawrence~Ein and Robert~Lazarsfeld, \textit{Vanishing theorems, a theorem of Severi, and the equations
defining projective varieties}, Journal of the Am.~Math.~Soc.~\textbf{4} (1991), 587--602.

\bibitem{Ein} Lawrence~Ein, \textit{The Ramification Divisors for Branched Coverings
of $\bP^n_k$}, Math.~Ann \textbf{261} (1982), 483--485.

\bibitem{EL} Lawrence~Ein and Robert Lazarsfeld, \textit{Syzygies and Koszul cohomology of smooth
projective varieties of arbitrary dimension}, Invent.~math. \textbf{111} (1993), 51--67.

\bibitem{EFS} David~Eisenbud, Gunnar~Fl\o ystad and Frank-Olaf~Schreyer, \textit{Sheaf Cohomology and
Free Resolutions over Exterior Algebras}, Trans.~of~the~AMS \textbf{355} (2003), 4397--4426.

\bibitem{EisGoto} David~Eisenbud and Shiro~Goto, \textit{Linear free resolutions and minimal multiplicity}, 
J.~Algebra~\textbf{88} (1984), 89--133.

\bibitem{GLP} Laurent~Gruson, Robert~Lazarsfeld and Christian~Peskine, \textit{On a theorem of Castelnuovo and the equations
defining space curves}, Invent.~math.~\textbf{72} (1983), 491--506.

\bibitem{Kwak} Sijong~Kwak, \textit{Castelnuovo regularity for smooth subvarieties of dimensions $3$ and $4$}, 
J.~Algebraic Geometry~\textbf{7} (1998), 195--206.

\bibitem{KwakE} Sijong~Kwak and E.~Park, \textit{Some effects of property $N_p$ on the higher normality
and defining equations of nonlinearly normal varieties}, 
J.~Reine~Angew.~Math.~\textbf{582} (2005), 87--105.

\bibitem{Laz1} Robert Lazarsfeld, \textit{A sharp Castelnuovo bound for smooth surfaces}, 
Duke~Math.~J.~\textbf{55} (1987), 423--429.

\bibitem{Laz2} Robert Lazarsfeld, \textit{Positivity in Algebraic Geometry}, 
Ergebnisse der Mathematik und ihrer Grenzgebiete Vol.~48/49, Springer Verlag, New York, 2004.

\bibitem{McCP} Jason~McCullough and Irena~Peeva, \textit{Counterexamples to the Eisenbud-Goto
regularity conjecture}, J.~Am.~Math.~Soc.~\textbf{31} (2018), 473--496.

\bibitem{Noma} Atsuki~Noma, \textit{Hypersurfaces cutting out a projective
variety}, Transactions of the Am. Math. Soc.~\textbf{362} (2010), 4481--4495.

\bibitem{Niu} Wenbo~Niu and Jinhyung~Park, \textit{A Castelnuovo-Mumford regularity bound for scrolls}, 
J.~Algebra~\textbf{488} (2017), 388--402.

\bibitem{Ran1} Ziv Ran, \textit{Local differential geometry and generic projections of threefolds},
J.~Diff.~Geometry~\textbf{32} (1990), 131--137.

\bibitem{Ran2} Ziv Ran, \textit{Unobstructedness of filling secants and the Gruson-Peskine general projection theorem},
Duke~Math.~J.~\textbf{164} (2015), 697--722.

\bibitem{Rat} Jürgen~Rathmann, \textit{On the Completeness of Linear Series cut out by Hypersurfaces}.

\end{thebibliography}
\end{document}